\documentclass[letterpaper, 10 pt, conference]{ieeeconf}

\IEEEoverridecommandlockouts                              
 \pdfoutput=1
\usepackage{amsmath, amssymb}
\usepackage{amsfonts}
\usepackage{graphicx}
\usepackage{enumerate}
\usepackage{caption}
\usepackage{ifthen}
\usepackage{multicol}
\usepackage{float}
\usepackage{cite}
\usepackage[usenames, dvipsnames]{color}
\usepackage[lofdepth,lotdepth]{subfig}
\graphicspath{{figures/}}
\usepackage{url}

\newtheorem{theorem}{Theorem}[section]

\newtheorem{proposition}[theorem]{Proposition}
\newtheorem{corollary}[theorem]{Corollary}
\newtheorem{rem}{Remark}

\newboolean{showcomments}
\setboolean{showcomments}{true}

\newcommand{\henrik}[1]{\ifthenelse{\boolean{showcomments}}
{\textcolor{Blue}{(Henrik says: #1)}}{}}
\newcommand{\emma}[1]{\ifthenelse{\boolean{showcomments}}
{\textcolor{Green}{(Emma says: #1)}}{}}

\newboolean{shownew}
\setboolean{shownew}{true}
\newcommand{\newtext}[1]{\ifthenelse{\boolean{shownew}}
{\textcolor{Red}{#1}}{}}

\newcommand{\hn}{$\mathcal{H}_2$ }

\usepackage{array}
\newcolumntype{L}[1]{>{\raggedright\let\newline\\\arraybackslash\hspace{0pt}}m{#1}}
\newcolumntype{C}[1]{>{\centering\let\newline\\\arraybackslash\hspace{0pt}}m{#1}}
\newcolumntype{R}[1]{>{\raggedleft\let\newline\\\arraybackslash\hspace{0pt}}m{#1}}
\usepackage{subfig}

\title{\LARGE \bf
Performance metrics for droop-controlled microgrids with variable voltage dynamics}

\author{Emma Tegling, Dennice F. Gayme and Henrik Sandberg %
\thanks{E. Tegling and H. Sandberg are with the School of Electrical Engineering and the ACCESS Linnaeus Center, KTH Royal Institute of Technology, SE-100 44 Stockholm, Sweden {(\tt tegling, hsan@kth.se)}. D. F. Gayme is with the Department of Mechanical Engineering at the Johns Hopkins University, Baltimore, MD, USA, 21218 {(\tt dennice@jhu.edu)}. Funding support from the NSF (National Science Foundation) through grant number ECCS 1230788 (D.F.G.) and the Swedish Research Council through grant 2013-5523 as well as the Swedish Foundation for Strategic Research through the project ICT-Psi (H.S.) is gratefully acknowledged. }  
 }%

\begin{document}
\maketitle
\thispagestyle{empty}
\pagestyle{empty}

\begin{abstract}
This paper investigates the performance of a microgrid with droop-controlled inverters in terms of the total power losses incurred in maintaining synchrony under persistent small disturbances. The inverters are modeled with variable frequencies and voltages under droop control. For small fluctuations from a steady state, these transient power losses can be quantified by an input-output \hn norm of a linear system subject to distributed disturbances. We evaluate this \hn norm under the assumption of a dominantly inductive network with identical inverters. The results indicate that while phase synchronization, in accordance with previous findings, produces losses that scale with a network's size but only weakly depend on its connectivity, the losses associated with the voltage control will be larger in a highly connected network than in a loosely connected one. The typically higher rate of convergence in a highly interconnected network thus comes at a cost of higher losses associated with the power flows used to reach the steady state.    

\end{abstract}

\section{Introduction}
\label{sec:intro}
Electric power generation is becoming increasingly distributed as the penetration of renewable energy sources increases \cite{smart_report2012}. Deferred investments in grid infrastructure along with high fossil fuel prices also make utilities incapable of meeting increased local demand centrally \cite{Farhangi2010}. Distributed generation (DG) units are typically connected to low to medium voltage grids via DC/AC or AC/AC power converters (inverters). Replacing the traditional centralized synchronous generator based power plants with these resources is leading to a much more heterogeneous power system. The \emph{microgrid} has been proposed as a key strategy to address issues related to this heterogeneity, as it allows for local operation of networks composed of loads and DG units, independently from the main grid \cite{Markvart2006}, \cite{Lasseter2002}. 

A key concern in the operation of micgrogrids is the control of the DG unit power inverters to ensure stability, power balance and synchronization \cite{Lopes2006}. A widely proposed control scheme in this context is \textit{droop control}, which is a decentralized proportional controller. A recent research trend is to characterize conditions for synchronous stability in microgrids under droop control. In particular, a series of work \cite{SimpsonPorco2012, SimpsonPorco2013} derives analytical conditions for synchronization and power sharing in droop-controlled inverter networks by drawing connections to models of coupled Kuramoto oscillators. Analogies between power systems and Kuramoto oscillators were previously used by D\"orfler and Bullo in \cite{Dorfler2010, Dorfler2011} to derive conditions for stability in synchronous generator networks.  While most of these works have assumed constant voltage profiles, there is a recent focus on analyzing droop control also for voltage and reactive power stabilization in microgrids \cite{Gentile2014, LuChu2013, Schiffer2014, Schiffer2014b}. In particular, Schiffer et al. \cite{Schiffer2014} derive conditions on gains for frequency and voltage stability by formulating a port-Hamiltonian description of the system with variable voltage dynamics.  

In the present work, we study the same type of inverter-based microgrids with variable frequencies and voltages under droop control as in \cite{Schiffer2014}. However, the question of concern here is not one of stability but rather of performance. We consider performance in terms of the total transient power losses incurred in maintaining synchrony under persistent small disturbances. These losses are associated with power flows that occur spontaneously in the system when an inverter is deviating from its nominal phase and voltage, and can be regarded as a measure of control effort. These transient power losses can be quantified through an input-output \hn norm of a linear time-invariant (LTI) system of coupled inverter dynamics, subject to distributed disturbances.

Conceptually, this performance measure relates to measures of disorder or robustness in general networks with consensus-type dynamics. Several such performance measures and their asymptotic scalings in large-scale networks were evaluated in \cite{Bamieh2012}. In the context of oscillator networks, robustness with respect to stochastic disturbances was studied in \cite{SiamiMotee2013} and the use of control nodes to improve performance in terms of inter-nodal interactions was recently evaluated in \cite{Grunberg2014}. 

The present study extends the work in \cite{BamiehGayme2012, Tegling2014}, in which performance in terms of power losses was evaluated for networks of synchronous generators, to include the effect of variable voltages. A similar approach was taken in \cite{Sjodin2014} to characterize these losses in heterogeneous networks with inverters. Such networks typically require tighter voltage control than e.g. transmission grids, and here we quantify the additional transient losses arising through fluctuating voltages. 

In the present work, we analyze the performance of the droop-controlled inverter network in two steps. First, we consider the network under the assumption of decoupled active and reactive power flows, resulting in decoupled frequency and voltage dynamics. We then re-introduce the cross-couplings and show that their effect is small compared to the overall transient losses. Our main result shows that under the assumption of uniform inverter parameters and resistance-to-reactance ratios in the network, the power losses can be decomposed into two parts; one associated with frequency control which is identical to previous results for synchronous generator networks, and an additional part associated with voltage control. The result reduces to the first part if voltages are held constant, in which case losses grow unboundedly with the network size, but are independent of network topology. The losses associated with variable voltages, however, are shown not only to grow with network size, but also to increase with increasing network connectivity. This means that although a more strongly connected network may have a higher rate of convergence \cite{Tang2011} and be easier to synchronize \cite{Dorfler2010}, this benefit comes at the cost of higher transient power losses.

The remainder of this paper is organized as follows. In Section \ref{sec:problem_setup} we introduce the model for the inverters and power flows. Section~\ref{sec:perfmeas} introduces the performance measure, which is evaluated in Section~\ref{sec:lossless} under the assumption of decoupled, lossless, network dynamics. In Section~\ref{sec:lossygrids} we discuss the effects of nonzero cross-couplings on the system performance before we conclude in Section~\ref{sec:conclusions}. 


\section{Problem setup}
\label{sec:problem_setup}
Consider a network $\mathcal{G} = \{\mathcal{N}, \mathcal{E}\}$ with the set of nodes $\mathcal{N} = \{1, \ldots, N\}$ and a set of edges, or network lines, $\mathcal{E} = \{\mathcal{E}_{ik}\}$. Each of these lines is represented by a constant admittance $y_{ik} = g_{ik} - \mathbf{j}b_{ik}$. Throughout this paper, we will assume a Kron-reduced network model (see e.g. \cite{Varaiya1985}), where loads are modeled as constant impedances that are absorbed into the network lines. Consequently, every node $i \in \mathcal{N}$ represents a generation unit with a power inverter as its grid interface. Each node has an associated phase angle $\delta_i$ and voltage magnitude $V_i$.
\begin{rem}
The current modeling framework can also allow for loads modeled as frequency-dependent active power withdrawals from the network. This approach would result in a network topology preserving system model as first proposed by Bergen and Hill \cite{BergenHill1981}. Previous results reported in \cite{Sjodin2014} suggest that such load models would not in principle alter the scaling properties of the losses studied here. 
\end{rem} 
 
\subsection{Inverter and droop control model}
We now introduce the models of the power inverters adopting the framwork presented in \cite{Schiffer2014}. We assume that these inverters are voltage sources, whose amplitude and frequency output can be regulated according to:
\begin{equation}
\label{eq:control1}
\begin{aligned}
\dot{\delta}_i &= u_i^\delta \\
\tau_{V_i} \dot{V}_i &= -V_i + u_i^V, 
\end{aligned}
\end{equation}
where $u_i^V$ and $u_i^\delta$ are the respective control signals. Here we have assumed that the voltage regulation is subject to a lag represented by a filter with time constant $\tau_{V_i} \ge 0$. 

The controls $u_i^\delta$ and $u_i^V$ are then implemented as simple proportional controllers (``droop controllers'') based on active and reactive power deviations respectively, which are given by:
\begin{equation}
\label{eq:droops}
\begin{aligned}
u_i^\delta &= \omega^* - k_{P_i}(\hat{P}_i - P_i^*)\\  
u_i^V &= V_i^* - k_{Q_i}(\hat{Q}_i - Q_i^*),
\end{aligned}
\end{equation}
where $\omega^*$, $V^*_i$, $P^*_i$ and $Q_i^*$ are the respective setpoints for the frequency, voltage magnitude, active and reactive power. The parameters $k_{P_i},~ k_{Q_i} >0$ are the respective droop coefficients. $\hat{P}_i$ and $\hat{Q}_i$ are the active and reactive powers measured by the power electronics at the inverter. These measurements are assumed to be processed through low-pass filters given by:
\begin{equation}
\label{eq:PQfilters}
\begin{aligned}
\tau_{P_i} \dot{\hat{P}}_i & = -\hat{P}_i + P_i \\
\tau_{Q_i} \dot{\hat{Q}}_i & = -\hat{Q}_i + Q_i,
\end{aligned}
\end{equation}
where $\tau_{P_i},~\tau_{Q_i} > 0$ are the filter time constants and $P_i$ and $Q_i$ are the actual power injections to the network at node $i$.
 
We can now use  \eqref{eq:control1} -- \eqref{eq:PQfilters} to formulate a closed-loop system. For this purpose, we first assume that the time constant for the voltage control, $\tau_{V_i}$ is small compared to $\tau_{Q_i}$, and can be neglected \cite{Schiffer2014}. We therefore set $\tau_{V_i} = 0$ in \eqref{eq:control1}, and then by substituting \eqref{eq:droops} into \eqref{eq:control1}, we obtain:
\begin{align}
 \dot{\delta_i} &= \omega_i \\ \label{eq:omegaalg}
 \omega_i &= \omega^* - k_{P_i}(\hat{P}_i - P_i^*)\\ \label{eq:Valg}
  V_i &= V_i^* - k_{Q_i}(\hat{Q}_i - Q_i^*),
  \end{align}  where we have introduced the inverter frequency $\omega_i$. Taking the derivatives of \eqref{eq:omegaalg} and \eqref{eq:Valg} with respect to time gives $\dot{\omega}_i = -k_{P_i} \dot{\hat{P}}_i$ and $\dot{V}_i = -k_{Q_i} \dot{\hat{Q}}_i$, in which we can insert equations \eqref{eq:PQfilters}. We then substitute $\hat{P}_i$ and $\hat{Q}_i$ using \eqref{eq:omegaalg} and \eqref{eq:Valg} and obtain the control dynamics for the phase angle and voltage as:
 \begin{equation}
 \label{eq:inverterdynamics}
 \begin{aligned}
  \dot{\delta_i} &= \omega_i\\ 
  \tau_{P_i}\dot{\omega}_i &= -\omega_i + \omega^* - k_{P_i}(P_i - P_i^*) \\
  \tau_{Q_i} \dot{V}_i & = -V_i + V_i^* -  k_{Q_i}(Q_i - Q_i^*).
 \end{aligned}
 \end{equation}
In the next section, we present the equations for $P_i$ and $Q_i$.

\subsection{Power flows}
\label{sec:powerflows}
Introducing $\delta_{ik} = (\delta_i - \delta_k)$ as the phase angle difference between neighboring nodes, the active and reactive powers injected to the grid at node $i \in \mathcal{N}$ are given by
\begin{align}
\label{eq:Pinj}
P_{i} &=  -g_{ii} V_i^2 + \sum_{k\sim i}V_iV_k(g_{ik} \cos\delta_{ik} + b_{ik}\sin\delta_{ik})\\ \label{eq:Qinj}
Q_{i} &= b_{ii} V_i^2 + \sum_{k \sim i}V_iV_k(g_{ik} \sin\delta_{ik} - b_{ik} \cos\delta_{ik}).
\end{align}
Here, $k \sim i $ indicates the existence of a line $\mathcal{E}_{ik}$ with associated conductance $g_{ik}$ and suceptance $b_{ik}$. At node $i$, $g_{ii} = \bar{g}_i + \sum_{k\sim i }g_{ik}$ and $b_{ii} = \bar{b}_i +\sum_{k \sim i} b_{ik}$ represent the respective shunt conductance and shunt susceptance. We will in the following make the common assumption \cite{Glover2011, Schiffer2014b} that the shunt elements are purely inductive, so that in our notation $\bar{g}_i=0$ and $\bar{b}\ge 0$ for all $i \in \mathcal{N}$. 

As per convention in power flow analysis, we assume that all quantities in equations \eqref{eq:Pinj} -- \eqref{eq:Qinj} have been normalized by system constants and are measured in per unit (p.u.). Throughout the paper, we will be considering the system under the assumption of small deviations from an operating point. For these reasons, we can approximate the power flows by a linearization around the point $ \textstyle P_{i}^0(\delta_{ik}^0, V_i^0, V_k^0)$ and $\textstyle Q_{i}^0(\delta_{ik}^0, V_i^0, V_k^0)$,  where $\textstyle V_i^0 = \textstyle V_k^0 =\textstyle V^0 = 1$ and $\delta_{ik} = 0$ for all $\textstyle i,k \in \mathcal{N}$.  This procedure gives the linearized power injections at node $i$ as:
\begin{align} \label{eq:Plin} 
\Delta P_i =& \sum_{k \sim i} \left( - g_{ik}(\Delta V_i - \Delta V_k) + b_{ik} \Delta \delta_{ik} \right)\\ \label{eq:qlin}
\Delta Q_i =& 2\bar{b}_i\Delta V_i +\sum_{k \sim i} \left( b_{ik}(\Delta V_i - \Delta V_k) + g_{ik} \Delta \delta_{ik}\right).
\end{align}
To simplify the remaining notation, we introduce the network admittance matrix $Y\in \mathbb{C}^{N \times N}$, given by $Y_{ii} = y_{ii}$ if $k = i$, $Y_{ik} = -y_{ik}$ if $k\sim i, ~k \neq i$ and zero otherwise. 
The matrix $Y$ can be partitioned into a real and an imaginary part:
\begin{equation} \label{eq:defY}
Y = L_G - \mathbf{j}(L_B+ \mathrm{diag}\{\bar{b}_i\}) , 
\end{equation} 
where $L_G$ denotes the network's conductance matrix and $L_B$ its susceptance matrix. By definition, the matrices $L_B$ and $L_G$ are weighted graph Laplacians of the network graphs defined respectively by the suseptances $b_{ij}$ and conductances  $g_{ij}$ of the network lines.

\subsection{LTI system formulation }
\label{sec:LTIdef}
We now formulate the inverter dynamics as a linear time-invariant (LTI) system subject to distributed disturbances acting on the inverters, representing fluctuations in generation and loads. For this purpose, we let the operating point around which the power flow equations are linearized be given by the setpoints introduced in \eqref{eq:droops}, such that $\Delta P_i = P_i - P_i^*$ and $\Delta Q_i = Q_i- Q_i^*$ for all $i \in \mathcal{N}$. Without loss of generality we then assume that the states at this point are transfered to the origin. In an effort to avoid cumbersome notation, we then omit the difference operator $\Delta$ and let the state variables $(\delta_{ik},~\omega_i,~ V_i )$ represent deviations from the operating point and assume additive process noise through the disturbance input $\mathrm{w}$.

We can then use the power flow equations \eqref{eq:Plin} -- \eqref{eq:qlin} to express the dynamics \eqref{eq:inverterdynamics} of the $i^{\mbox{th}}$ inverter as: 
\begin{small}
\begin{align}
\nonumber
\dot{\delta}_i &= \omega_i \\ \nonumber
\tau_{P_i}\dot{\omega}_i &= -\omega_i - k_{P_i}( - \sum_{k \sim i} g_{ik}( V_i -  V_k) + \sum_{k \sim i } b_{ik}  \delta_{ik}) + \mathrm{w}_i^\omega \\ \nonumber
  \tau_{Q_i} \dot{V}_i & = -V_i -  k_{Q_i}( 2\bar{b}_iV_i+ \sum_{k \sim i} b_{ik}( V_i -  V_k) + \sum_{k \sim i } g_{ik}  \delta_{ik} ) +\mathrm{w}_i^V.
\end{align} 
\end{small}
Now, by defining $\delta,~\omega, ~V$ as column vectors containing the states $\delta_i,~\omega_i, ~V_i$, $i \in \mathcal{N}$ and using the susceptance and conductance matrices defined in \eqref{eq:defY}, we can re-write the above in state space form as follows: 
\begin{small} 
\begin{align}  \nonumber
\label{eq:sseqns} 
\begin{bmatrix}
\dot{\delta} \\ \dot{\omega} \\ \dot{V}
\end{bmatrix}  =& \begin{bmatrix}
0 & I & 0 \\ -K_PT_P^{-1} L_B & -T_P^{-1} & K_PT_P^{-1} L_G \\-K_QT_Q^{-1} L_G & 0 & -C_Q T_Q^{-1} - K_QT_Q^{-1} L_B
\end{bmatrix} \begin{bmatrix}
\delta \\ \omega \\ V
\end{bmatrix}   \\
&+ \begin{bmatrix}
0 & 0 \\ T_P^{-1} & 0 \\ 0 & T_Q^{-1}
\end{bmatrix} \mathrm{w},
\end{align} \end{small} 
where $\mathrm{w} = [\mathrm{w}_i^\omega, ~\mathrm{w}_i^V]^T$ represents the disturbance input. We have also introduced $C_Q =\mathrm{diag}\{c_{Q_i} \} $ with $c_{Q_i} = 1 + 2k_{Q_i}\bar{b}_i $. The remaining system parameters are given by $K_{P/Q} = \mathrm{diag}\{k_{P/Q_i}\}$,  $T_{P/Q} = \mathrm{diag}\{\tau_{P/Q_i}\}$.

\section{System performance}
\label{sec:perfmeas}
In this paper, we are concerned with the performance of the system \eqref{eq:sseqns} in terms of the real power losses incurred in returning the system to a synchronous state following a disturbance, or in maintaining this state under persistent stochastic disturbances $\mathrm{w}$. These losses are associated with the power flows that arise when the network compensates for a node's deviating voltage or phase angle, and can be regarded as the control effort required to drive the system to a steady state. 

By defining an output of the system \eqref{eq:sseqns} as a measurement of the power losses associated with the system trajectories, we can evaluate the total transient losses using an input-ouptut \hn norm. Here we adopt the approach first introduced in \cite{BamiehGayme2012}, but extend the performance measure to also reflect non-uniform voltage profiles across the network. 

First, consider a general MIMO input-output system $G$ in state space form
\begin{equation}
\begin{aligned}
\label{eq:genericiosystem}
\dot{x} = & A x + b\mathrm{w} \\ y =& Cx.
\end{aligned}
\end{equation}
Provided that $G$ is stable, its squared \hn norm can be interpreted as the total steady-state variance of the output, when the input $\mathrm{w}$ is white noise with unit covariance, i.e., 
\[ ||G||_2^2 = \lim_{t \rightarrow \infty} \mathbb{E}\{y^*(t)y(t)\}. \]
In our case, by defining an output  $y(t)$ of \eqref{eq:sseqns} such that the total power losses satisfy $\mathbf{P}_{\mathrm{loss}}(t) = y^*(t)y(t)$, we will obtain the total expected power losses under a white noise disturbance input as the \hn norm from $\mathrm{w}$ to $y$ . 
\begin{rem}
Further standard interpretations of the \hn norm motivate the use of this norm to quantify the total power losses also under other input scenarios, see e.g. \cite{Tegling2014}.\end{rem} 

To define the relevant output measure, consider the real power loss over an edge $\mathcal{E}_{ik}$, given by Ohm's law as
\begin{equation} 
\label{eq:ohmflow}
P_{ik}^{\mathrm{loss}} = g_{ik}|v_i - v_k|^2, 
\end{equation}
where $v_i, ~ v_k$ are the complex voltages at nodes $i$ and $k$ (i.e. $\textstyle v_i = \textstyle V_ie^{\mathbf{j}\delta_i}$). We can now enforce the common linearized system assumption of small phase angle differences. Standard trigonometric methods then give that $\textstyle |v_i - v_k|^2 \approx \textstyle (V_i - V_k)^2 + (V_i(\delta_i - \delta_k))^2$. Since we also assume $\textstyle V_i \approx 1$ p.u. around the linearization point for all $\textstyle i \in \mathcal{N}$, an approximation of the power loss over $\mathcal{E}_{ik}$ is $ \textstyle  P_{ik}^{\mathrm{loss}} =  \textstyle g_{ik}\left[(V_i - V_k)^2 + (\delta_i - \delta_k) \right]^2$. The total instantaneous power losses over the network are then approximately
\begin{equation}
\label{eq:quadPloss}
\mathbf{P}_{\mathrm{loss}} = \sum_{i \sim k} g_{ik} \left[ \left( V_i  - V_k \right)^2 + \left( \delta_i -\delta_k \right)^2 \right].
\end{equation}
Making use of the conductance matrix $L_G$ defined in \eqref{eq:defY}, we can write \eqref{eq:quadPloss} as the quadratic form 
\begin{equation} \label{eq:quadformloss}
\mathbf{P}_{\mathrm{loss}} = V^TL_GV + \delta^T L_G \delta,
\end{equation}
where $V$ and $\delta$ are the state vectors defined in Section~\ref{sec:LTIdef}. 
Noting that $L_G$ is a positive semidefinite graph Laplacian and therefore has a unique positive semidefinite square-root $L_G^{1/2}$, we can define $y$ as
\begin{equation}
\label{eq:plossoutputdef}
y = \begin{bmatrix}
L_G^{1/2} & 0 & 0 \\ 0 & 0 & L_G^{1/2}
\end{bmatrix} \begin{bmatrix}
\delta \\ \omega \\ V
\end{bmatrix},
\end{equation}
which gives precisely that $\mathbf{P}^{\mathrm{loss}}= y^*y$. 

The performance measure \eqref{eq:quadformloss} represents the sum of the squared weighted differences in states between neighboring nodes, i.e., of \emph{local} state deviations.  It can therefore be regarded as a local, or microscopic, error measure which stands in contrast to other measures of disorder on a macroscopic level. Macroscopic measures could for example be each node's deviation from a network average or from a selected reference node. These types of performance measures and their asymptotic scalings in large networks were evaluated in \cite{Bamieh2012} for vehicular formations in regular network structures. Such vehicular formation problems, formulated as second order consensus dynamics, are very similar to the synchronization problem considered here. As we will also show, the scaling of the microscopic error measure considered there scaled with network size in the same manner as the total power losses obtained through \eqref{eq:quadformloss}. Meaningful measures of macroscopic disorder or ``coherence'' can also be defined in the context of synchronization in power networks, some of which are studied in a recent preprint \cite{Grunberg2015}.




\section{Analysis of decoupled microgrid dynamics}
\label{sec:lossless}
In this section, we analyze the dominant performance of \eqref{eq:sseqns} with respect to the output \eqref{eq:plossoutputdef}, by assuming that the network's resistances are small compared to its reactances. Under this common assumption, the active power flow is a function only of the phase angles and the reactive power flow is a function only of the voltage magnitudes, i.e., $P(\delta,V) \approx P(\delta)$, $Q(\delta,V) \approx Q(V)$, see e.g. \cite{Gentile2014, Schiffer2014, SimpsonPorco2015}. This leads to a decoupling of the frequency and voltage dynamics and we obtain $L_G = 0$ in the system matrix of \eqref{eq:sseqns}. The output \eqref{eq:plossoutputdef} then measures the power losses associated with the trajectories arising from these decoupled dynamics by retaining the non-zero resistances through $L_G$.


In Section \ref{sec:lossygrids}, we relax the assumption of decoupled power flows, and show that the results derived here are robust towards that relaxation, provided resistances remain sufficiently small. In particular, the errors made by evaluating the performance under lossless dynamics will be small in relation to the overall performance of the network.

\begin{rem}
The assumption of negligible or small resistances compared to reactances is not, in general, applicable to low to medium voltage grids \cite{Purchala2005}. However, it is not unreasonable for an inverter-based network, given that inverter output impedances are typically highly inductive \cite{Schiffer2014b}. When these are absorbed into the network through the Kron reduction they may therefore dominate the line resistances.  
\end{rem}



In the subsequent derivations we make the following further assumptions:
\begin{enumerate}[i)]
\item \textit{Identical inverters.} All inverters have identical droop control settings and low-pass filters for power measurement, i.e., $K_P = \mathrm{diag}\{k_P\}$, $K_Q = \mathrm{diag}\{k_Q\}$, $T_P = \mathrm{diag}\{\tau_P\}$, $T_Q = \mathrm{diag}\{\tau_Q\}$.
\item \emph{Uniform shunt conductances.} All nodes have identical shunt conductances, i.e., $\bar{b}_i = \bar{b} \ge 0$ and $C_Q = \mathrm{diag}\{c_Q\}$.
\item \emph{Uniform resistance-to-reactance ratios.} The ratio of resistances to reactances, equivalently conductances to susceptances, of all lines are uniform and constant, i.e., \[\alpha : = \frac{g_{ik}}{b_{ik}}, \] for all $\mathcal{E}_{ik} \in \mathcal{E}$. This implies $L_G= \alpha L_B$.
\end{enumerate}
Assumption (iii), which is also applied in e.g. \cite{LuChu2013}, \cite{Dorfler2014}, can be motivated first by uniformity in the physical line properties in a microgrid (i.e., materials and dimensions). Second, Kron reduction of a network increases its uniformity in node degrees \cite{Motter2013}. This also makes the line properties more uniform than in actual power networks. 

For ease of reference, we now re-state the system \eqref{eq:sseqns} with the output \eqref{eq:plossoutputdef} under assumptions (i) - (iii) as the multiple-input multiple-output (MIMO) LTI system $H$:
\begin{subequations}
\label{eq:iosystem1}
\begin{align}
\begin{bmatrix} \nonumber
\dot{\delta} \\ \dot{\omega} \\ \dot{V}
\end{bmatrix}  =& \begin{bmatrix}
0 & I & 0 \\ -\frac{k_P}{\tau_P} L_B & -\frac{1}{\tau_P} I & 0\\0& 0 & -\frac{c_Q}{\tau_Q} I -\frac{k_Q}{\tau_Q} L_B
\end{bmatrix} \begin{bmatrix}
\delta \\ \omega \\ V
\end{bmatrix}   \\ \label{eq:losslessstates}
&+ \begin{bmatrix}
0 & 0 \\ \frac{1}{\tau_P}I & 0 \\ 0 & \frac{1}{\tau_Q}I
\end{bmatrix} \mathrm{w}  =: A\psi + B \mathrm{w},\\  \label{eq:losslessoutput}
y =& \begin{bmatrix}
\sqrt{\alpha}L_B^{1/2} & 0 & 0 \\ 0 & 0 & \sqrt{\alpha}L_B^{1/2}
\end{bmatrix} \begin{bmatrix}
\delta \\ \omega \\ V
\end{bmatrix} =: C \psi.
\end{align}
\end{subequations}

The remainder of this section is organized as follows. We will first show that the system \eqref{eq:iosystem1} is input-output stable in order to ensure that its \hn norm is bounded. We then proceed to derive an expression for this norm and to state its value for specific network topologies. Finally, we discuss these results in relation to previous ones on coupled oscillator networks. 

\subsection{Eigenvalues and stability}
By definition, $L_B$ and $L_G$ are weighted graph Laplacians and as such they satisfy the equation 
\(L_B\mathbf{1} = L_G\mathbf{1} = \mathbf{0},\)
where $\mathbf{1}$ is the vector of all ones. It is easy to show that this zero eigenvalue is also an eigenvalue of the system \eqref{eq:iosystem1}, corresponding to the drift of the mean phase angle.  This mode is however unobservable from the output, as we will show by a simple state transformation in the following section. Now, if we denote by $\lambda_n^B$ the $n^{\mbox{th}}$ eigenvalue of $L_B$ and WLOG assume $\lambda_1^B = 0$. The remaining eigenvalues of \eqref{eq:iosystem1} are then given by Theorem~\ref{thm:eigenvalues}. 
\begin{theorem}
\label{thm:eigenvalues}
If the graph underlying the network $\mathcal{G}$ is connected, then the eigenvalues of the system \eqref{eq:losslessstates} are:
\begin{equation*}
\begin{split}
\Lambda(A) = \left\lbrace 0, -\frac{1}{2\tau_P}\right. & (1 \pm \sqrt{1-k_P\tau_P\lambda_n^B}),   \\
 & \left.  -\frac{1}{\tau_P}, -\frac{c_Q}{\tau_Q}, -\frac{c_Q}{\tau_Q} - \frac{k_Q}{\tau_Q}\lambda_n^B \right\rbrace , \end{split}
\end{equation*}
for $n = \{2,\ldots, N\}$. If the parameters $k_P,\tau_P,k_Q,\tau_Q >0$ and the shunt susceptance satisfies $ \textstyle \bar{b} > \textstyle\frac{-1}{2k_Q} \Leftrightarrow \textstyle c_Q>0$, all eigenvalues apart from $\lambda_1 = 0$ lie strictly in the left half of the complex plane and $A$ is a stable matrix. 
\end{theorem}
\begin{proof}
The eigenvalues are given by $A$'s characteristic polynomial. Since the graph underlying $L_B$ is connected,  $L_B$ is Hermitian positive semi-definite and $\textstyle 0 = \textstyle \lambda_1^B <\textstyle \lambda_2^B \le \textstyle \ldots \le \textstyle \lambda_N^B$. It is then easy to see that if $\textstyle k_P,\tau_P,k_Q,\tau_Q,c_Q >0$, all eigenvalues have negative real parts. 
\end{proof}
Using this result we conclude that the system \eqref{eq:iosystem1} is input-output stable and that its \hn norm is bounded. 

\subsection{\hn norm calculation}
To derive the \hn norm of \eqref{eq:iosystem1}, we follow the approach in \cite{BamiehGayme2012} and use the following unitary state transformation:
\[\begin{bmatrix}
\delta \\ \omega \\ V
\end{bmatrix} =: \begin{bmatrix}
U & 0 & 0\\ 0& U & 0\\ 0& 0& U
\end{bmatrix}\begin{bmatrix}
\hat{\delta} \\ \hat{\omega} \\ \hat{V}
\end{bmatrix},
\]
where $U$ is the unitary matrix which diagonalizes $L_B$, i.e., $L_B = U^*\Lambda_BU$ with $\Lambda_B = \mathrm{diag}\{\lambda_1^B, \lambda_2^B, \ldots, \lambda_N^B\}$. Given that the \hn norm is unitarily invariant, we can also apply transformations to the input and the output, such that $\hat{y} = \begin{bmatrix}
U^* & 0 \\ 0 & U^* 
\end{bmatrix}y$ and $\hat{\mathrm{w}} = \begin{bmatrix}
U^* & 0 \\ 0 & U^* 
\end{bmatrix}\mathrm{w}$. Since we have assumed $L_G = \alpha L_B$, we have that $L_G^{1/2}$ and $L_B$ are simultaneously diagonalizable. Therefore, $U^*L_G^{1/2}U = \Lambda_G^{1/2} = \sqrt{\alpha} \Lambda_B^{1/2}$.

Through these transformations, we obtain a system $\hat{H}$ in which all blocks of the system \eqref{eq:iosystem1} have been diagonalized. This system thus represents $N$ decoupled subsystems $\hat{H}_n$: 

\begin{align}
\nonumber
\begin{bmatrix}
\dot{\hat{\delta}}_n \\ \dot{\hat{\omega}}_n \\ \dot{\hat{V}}_n
\end{bmatrix}  =& \begin{bmatrix}
0 & 1 & 0 \\ -\frac{k_P}{\tau_P} \lambda^B_n & -\frac{1}{\tau_P} & 0\\0& 0 & -\frac{c_Q}{\tau_Q}  -\frac{k_Q}{\tau_Q} \lambda^B_n
\end{bmatrix} \begin{bmatrix}
\hat{\delta}_n \\ \hat{\omega}_n \\ \hat{V}_n
\end{bmatrix}   \\ \label{eq:decoupledsystem}
&+ \begin{bmatrix}
0 & 0 \\ \frac{1}{\tau_P} & 0 \\ 0 & \frac{1}{\tau_Q}
\end{bmatrix} \hat{\mathrm{w}}_n =: A_n \hat{\Psi}_n + B_n \hat{\mathrm{w}}_n, \\  \nonumber
\hat{y}_n =& \sqrt{\alpha \lambda_n^B} \begin{bmatrix}
1 & 0 & 0 \\ 0 & 0 & 1 
\end{bmatrix} \begin{bmatrix}
\hat{\delta}_n \\ \hat{\omega}_n \\ \hat{V}_n
\end{bmatrix} =: C_n \hat{\Psi}_n, 
\end{align}
and the \hn norm of the system $\hat{H}$ will be $||\hat{H}||_2^2 = \textstyle \sum_{n=1}^N ||\hat{H}_n||_2^2 = ||H||_2^2$. Notice that the subsystem $\hat{H}_1$ corresponding to $\lambda_1 = 0$ has the the output $\hat{y}_1 = 0\cdot\hat{\psi}$. It is therefore unobservable and has $||\hat{H}_1||_2^2 = 0$. 

The remaining subsystems' \hn norms are obtained by calculating the observability Gramian $X_n \in \mathbb{C}^{3\times 3}$ from the Lyapunov equation \begin{equation}
\label{eq:lyapunov}
A_n^*X_n + X_nA_n = -C_n^*C_n.\end{equation}  We then have that $ \textstyle ||\hat{H}_n||_2^2 = \textstyle \mathrm{tr}\{ B_n^*X_n^{\phantom{1}} B_n^{\phantom{1}}\} =\textstyle \frac{1}{\tau_P^2} X_{n_{22}} + \textstyle \frac{1}{\tau_Q^2} X_{n_{33}}$. Due to space limitations we omit the expansion of \eqref{eq:lyapunov}, but note its solution for $X_{n_{22}}$ and $X_{n_{33}}$: \[X_{n_{22}} = \frac{\alpha \tau_P^2}{2 k_P},~ ~X_{n_{33}} = \frac{\alpha \tau_Q}{2}\cdot \frac{1}{\frac{c_Q}{\lambda_n^B} + k_Q}.\]
Finally, summing up the $N-1$ non-zero subsystem norms leads to our main result:
\begin{theorem}
\label{thm:mainresult}
The squared \hn norm of the input-output mapping \eqref{eq:iosystem1} is given by:
\begin{equation}
\label{eq:plossnorm} \boxed{
||H||_2^2 = \frac{\alpha}{2 k_P}(N-1) + \frac{\alpha }{2 \tau_Q}\sum_{n=2}^N \frac{1}{\frac{c_Q}{\lambda_n^B} + k_Q}.}
\end{equation}
According to the discussion in Section~\ref{sec:perfmeas}, this expression represents the expected power losses due to a white noise disturbance input $\mathrm{w}$. 
\end{theorem}

Under the present assumptions, the inverter's frequency and voltage control dynamics are decoupled. Due to the decoupled output measurement, the \hn norm in \eqref{eq:plossnorm} can be shown to be the sum of the respective norms of two decoupled subsystems: $||H||_2^2 = ||H^\delta||_2^2 + ||H^V||_2^2$.

The power losses associated with the phase angle synchronization and active power sharing are then  \begin{equation}
\label{eq:phaselosses}
 ||H^\delta||_2^2 = \frac{\alpha}{2 k_P}(N-1).
\end{equation}
This is the same result as obtained for systems of synchronous generators in \cite{BamiehGayme2012}, where the droop coefficient $k_P$ is analogous to the generator damping. These losses scale linearly with the number of nodes $N$, in the same way as the microscopic error measures evaluated in \cite{Bamieh2012} for the vehicular formation problem. It is also worth noting that this quantity is entirely independent of network topology, i.e., a loosely connected network will incur the same transient losses during phase synchronization as a highly connected one. While this conclusion only holds under the present assumptions of uniform generator parameters and equal resistance-to-reactance ratios $\alpha$ across the network, it has been demonstrated that the topology dependence remains weak when those assumptions are relaxed \cite{Tegling2014}. We also remark that Siami and Motee in \cite{SiamiMotee2013} propose a network-weighted average $\bar{\alpha}$ which provides a generalization of \eqref{eq:phaselosses}. 

The losses associated with the voltage control and reactive power sharing are given by 
\begin{equation}
\label{eq:voltagelosses}
||H^V||_2^2 = \frac{\alpha }{2 \tau_Q}  \sum_{n=2}^N \frac{1}{\frac{c_Q}{\lambda_n^B} + k_Q},
\end{equation}
and depend on the topology of the network through the eigenvalues $\lambda_n^B$ of $L_B$. The losses increase when the eigenvalues $\lambda_n^B$ are larger, which implies that they increase with increasing line susceptances and network connectivity. These losses can be said to be inversely related to what we may call the network's \emph{total effective reactance} as studied in \cite{Grunberg2015}, but we defer further discussion of this notion to future work. 

\subsection{Specific network topologies}
\label{sec:topologies}
The result of Theorem \ref{thm:mainresult} indicates that transient losses increase with increasing network connectivity. While microgrid network structures may vary, in terms of connectivity they all fall somewhere between the two extremes given by the complete graph and the line graph. We next present results for these two special cases.

\begin{theorem}
\label{thm:complete}
If the graph underlying the network $\mathcal{G}$ is \textit{complete}, i.e., there is a line $\mathcal{E}_{ik}$ connecting each node pair $i,k\in \mathcal{N}$, then the expected power losses are bounded from above by:
\begin{equation}
\label{eq:completebound}
 ||H||_2^2 \le \frac{\alpha}{2} (N-1) \left( \frac{1}{k_P} + \frac{1}{\tau_Q \left(\frac{c_Q}{N\underline{b}} +k_Q \right)} \right), 
\end{equation}
where $\underline{b}$ is the arithmetic mean of the susceptances $b_{ik}$ for all network lines $\mathcal{E}_{ik} \in \mathcal{E}$. 

The losses are bounded from below by:
\begin{equation}
\label{eq:completeboundlower}
||H||_2^2  \ge \frac{\alpha}{2} (N-1) \left( \frac{1}{k_P} + \frac{1}{\tau_Q \left(\frac{c_Q}{N b_{\min}} +k_Q \right)} \right), 
\end{equation}
where $b_{\min} = \min_{\mathcal{E}}b_{ik}$. If $b_{ik} = \underline{b} = b_{\min}$ for all $\mathcal{E}_{ik} \in \mathcal{E}$, \eqref{eq:completebound} -- \eqref{eq:completeboundlower} turn into equalities.
 \end{theorem} 
\begin{proof}
See Appendix.
\end{proof}
\begin{corollary}
\label{cor:complete}
If the graph underlying the network $\mathcal{G}$ is complete, then for large $N$
\begin{equation}
\label{eq:completelimit}
||H||_2^2\approx \frac{\alpha}{2}(N-1)\left( \frac{1}{k_P} + \frac{1}{\tau_Qk_Q} \right).  \end{equation}
\end{corollary} \begin{proof}
For large $N$, $\textstyle \frac{c_Q}{N\underline{b}} \rightarrow 0$ and $\textstyle \frac{c_Q}{Nb_{\min}} \rightarrow 0$ and the result follows. 
\end{proof}
By Corollary~\ref{cor:complete}, the losses in a large fully connected network will depend on the droop settings for active and reactive power respectively, where higher droop gains give smaller losses. We also notice that the losses associated with the voltage control decrease with increasing $\tau_Q$. In the limit where $\tau_Q \rightarrow \infty$, the voltages are constant, and we retrieve the result from \cite{BamiehGayme2012}, in which a constant voltage profile was an underlying modeling assumption. In any case, the losses will grow unboundedly with the network size $N$. 
\begin{figure}
\centering
\includegraphics[width = 0.45\textwidth]{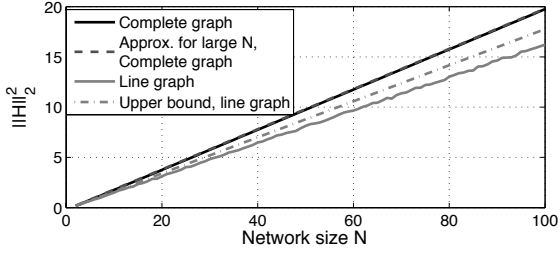}
\caption{Values of the \hn norm in \eqref{eq:plossnorm} for sample networks with line graph and complete graph structure, along with the approximation \eqref{eq:completelimit}, whose line coincides with the the complete graph line, and the bound \eqref{eq:boundradial}. Here, $k_P = k_Q = c_Q = 1$, $\alpha = 0.2$ and the line susceptances are uniformly distributed on the interval $(0.5,3.25)$.  }
\label{fig:scaling}
\end{figure}

In this paper, we consider a Kron reduced network model. Such reductions of power networks in general result in fully connected effective networks \cite{Motter2013}, and the expressions \eqref{eq:completebound} -- \eqref{eq:completelimit} hold. However, future microgrids may arise through the addition of generation units at some or all nodes in distribution grids. Distribution grids typically have a radial network structure, i.e., have a line graph as their underlying topology, and would maintain line graph structure also in the Kron-reduced case. The following theorem describes the transient losses in this case:
\begin{theorem}
\label{thm:line}
If the graph underlying the network $\mathcal{G}$ is a \textit{line graph}, i.e., $\mathcal{E}  = \{ \mathcal{E}_{i,i-1}, \mathcal{E}_{i,i+1} \}$ for $i = 2,\ldots, N-1$ and $\underline{b}$ is the arithmetic mean of the associated line susceptances, then the power losses are bounded by:
\begin{equation}
\label{eq:boundradial}
||H||_2^2  \le \frac{\alpha}{2} (N-1) \left( \frac{1}{k_P}  + \frac{1}{\tau_Q(\frac{c_Q}{2\underline{b}}+k_Q) }\right).  \end{equation} 
\end{theorem} 
\begin{proof}
See Appendix.
\end{proof}
We notice that even for large $N$, these losses will depend on the actual value of the average line susceptance $\underline{b}$, in contrast to the result for complete graphs in Corollary~\ref{cor:complete}. The underlying scaling of the losses with the network size $N$, however, remains. In Fig.~\ref{fig:scaling} the values of the \hn norm as a function of network size $N$ are displayed for the two network topologies discussed in this section.

The fact that a highly interconnected network incurs larger power losses in recovering or maintaining synchrony than a loosely connected network stands in contrast to typical notions of power system stability. For example, it has been shown that highly interconnected networks are easier to synchronize \cite{Dorfler2010, Schiffer2014} and have a higher rate of convergence \cite{Tang2011}. Fig.~\ref{fig:simulation} shows the transient behaviors obtained from simulations of a 5 node network with respective complete and line graph topologies. The plot clearly shows a faster convergence in the complete graph case. This faster convergence, however, comes at a greater cost in terms of power losses. An intuition behind this result may be to consider the non-equilibrium power flows as additional controllers in the network, so that the associated losses are a measure of their control effort. Then, more links imply that the voltage drop which occurs between each node pair leads to a larger overall number of power flows, i.e., more control, and thus more losses. 

\begin{figure*}
\centering
\includegraphics[width = 1\textwidth]{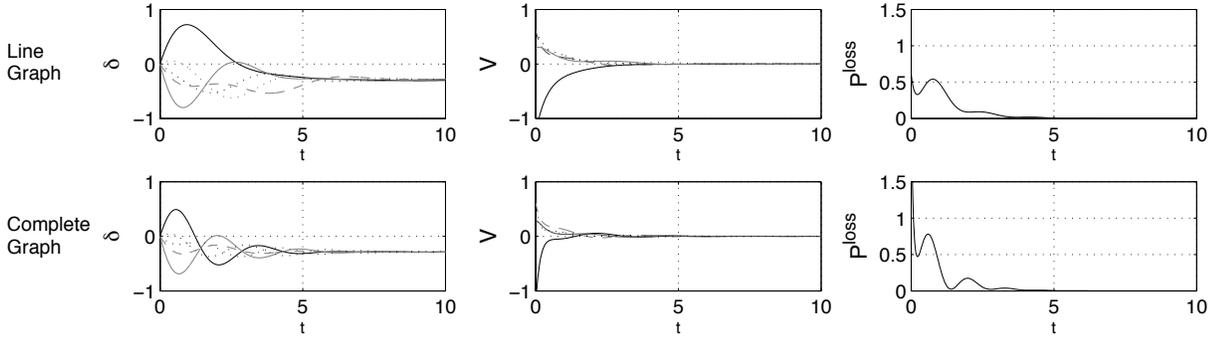}
\caption{Simulations of the system \eqref{eq:iosystem1} with $N=5$ inverters.}
\label{fig:simulation}
\end{figure*}

\section{Performance with cross-coupled voltage and frequency dynamics}
\label{sec:lossygrids}
We will now relax the assumption of decoupled microgrid dynamics and again study the system \eqref{eq:sseqns}. Using assumptions (i) -- (iii) of Section~\ref{sec:lossless} we can formulate the MIMO system $H^\alpha$:
\begin{subequations}
\label{eq:iosystem2}
\begin{align}
\begin{bmatrix} \nonumber
\dot{\delta} \\ \dot{\omega} \\ \dot{V}
\end{bmatrix}  =& \begin{bmatrix}
0 & I & 0 \\ -\frac{k_P}{\tau_P} L_B & -\frac{1}{\tau_P} I & \frac{k_P}{\tau_P}\alpha L_B\\ -\frac{k_Q}{\tau_Q}\alpha L_B & 0 & -\frac{c_Q}{\tau_Q} I -\frac{k_Q}{\tau_Q} L_B
\end{bmatrix} \begin{bmatrix}
\delta \\ \omega \\ V
\end{bmatrix}   \\ \label{eq:lossystates}
&+ \begin{bmatrix}
0 & 0 \\ \frac{1}{\tau_P}I & 0 \\ 0 & \frac{1}{\tau_Q}I
\end{bmatrix} \mathrm{w}  \\  \label{eq:lossyoutput}
y =& \begin{bmatrix}
\sqrt{\alpha} L_B^{1/2} & 0 & 0 \\ 0 & 0 & \sqrt{\alpha} L_B^{1/2}
\end{bmatrix} \begin{bmatrix}
\delta \\ \omega \\ V
\end{bmatrix}.
\end{align}
\end{subequations}
Compared to the lossless dynamics in \eqref{eq:losslessstates}, the system matrix in \eqref{eq:lossystates} has cross-couplings between the voltage and frequency dynamics which are proportional to the resistance-to-reactance ratio $\alpha$. We will examine the effect of these cross-couplings on the system's performance in terms of the cross-coupling strength $\alpha$. In particular, we are interested in characterizing the error obtained through the assumption of lossless microgrid dynamics from Section~\ref{sec:lossless}. 
\begin{figure}
\centering
\includegraphics[width = 0.45\textwidth]{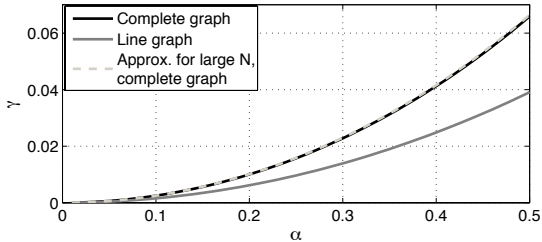}
\caption{Norm error $\gamma$ in \eqref{eq:defreldiff} as a function of $\alpha$ for networks of size $N=50$ with complete graph and line graph structure, along with the approximation \eqref{eq:gamma_complete}. Here, $x_{ij} = 0.2$, $k_P = 1$, $k_Q=2$, $\tau_P = \tau_Q = 0.5$ and $c_Q=1$.   }
\label{fig:reldiff}
\end{figure} 

Consider for this purpose the relative error in the squared \hn norm between the system $H^\alpha$ in \eqref{eq:iosystem2} and the  decoupled system $H$ in \eqref{eq:iosystem1}:
\begin{equation}
\label{eq:defreldiff}
\gamma = \frac{||H^\alpha||_2^2 - ||H||^2_2 }{||H||_2^2}.
\end{equation}
This quantity can be evaluated numerically and is shown in Figure~\ref{fig:reldiff} for $\alpha \in (0.01,0.5)$ for two sample networks of size $N=50$. We observe that the error is small and decreases faster than linearly as $\alpha \rightarrow 0$. These observations are accounted for by the following proposition:

\begin{proposition}
\label{prop:fullnorm}
The squared \hn norm of the system $H^\alpha$ in \eqref{eq:iosystem2} is, for sufficiently small $\alpha$, given by:
 \begin{equation}
\label{eq:proposednorm}
||H^\alpha||_2^2 = c_1(\Lambda_B) \alpha + c_2(\Lambda_B)\alpha^3 + c_3(\Lambda_B)(\alpha^5) + \ldots,
\end{equation}
where $c_k(\Lambda_B)$, $k = 1,2,\ldots$, are scalar functions of the eigenvalues of $L_B$. For the first term it holds that
\[c_1(\Lambda_B)\alpha = ||H||_2^2, \]
where $||H||_2^2$ was given in Theorem~\ref{thm:mainresult}. Hence,
\[\gamma =\frac{c_2(\Lambda_B) }{c_1(\Lambda_B) } \alpha^2 + \frac{c_3(\Lambda_B) }{c_1(\Lambda_B) } \alpha^4 +\ldots . \]
\end{proposition} \vspace{2mm}
\begin{proof}
The result is obtained in a manner analogous to the derivation in Section~\ref{sec:lossless}. Due to space limitations the full details are omitted here. 
\end{proof}
Proposition~\ref{prop:fullnorm} shows that the results obtained by assuming a lossless microgrid with decoupled dynamics are robust in the sense that the error is proportional to higher order powers of the coupling strength $\alpha$, provided $\alpha$ is small enough to guarantee the boundedness of $||H^\alpha||_2^2$. 

Now, consider again the special case where the graph underlying the network $\mathcal{G}$ is complete. If the number of nodes is large, the coefficients in \eqref{eq:proposednorm} are given as
\[c_k(\Lambda_B) = (N-1)\frac{k_P+k_Q\tau_Q}{2k_Q^2} \left(\frac{k_P\tau_Q}{k_Q} \right)^{k-2}, \]
for $k = 2, 3, \ldots$, and for all $\alpha$ such that $\frac{k_P\tau_Q}{k_Q} \alpha^2 <1$. The coefficient $c_1(\Lambda_B)$ is given by Corollary~\ref{cor:complete}.  

The relative error $\gamma$ for the complete graph, then satisfies  \begin{equation} \label{eq:gamma_complete}
\gamma =  \frac{k_P\tau_Q}{k_Q}\alpha^2+ \left( \frac{k_P\tau_Q}{k_Q}\right)^2 \alpha^4 + \ldots.  \end{equation}
Numerical results indicate that \eqref{eq:gamma_complete} also provides an upper bound for the relative error \eqref{eq:defreldiff} for general network topologies, as seen in Fig.~\ref{fig:reldiff}. 

The result \eqref{eq:gamma_complete} shows that faster voltage control (large $k_Q$, small $\tau_Q$) will decrease the effect of the cross-couplings on the transient power losses. Since large droop coefficients also decrease the overall losses by Theorem~\ref{thm:mainresult}, one may wish to prioritize a large voltage droop setting $k_Q$ when tuning power inverters in low to medium voltage grids where resistances are non-negligible.

\section{Summary and conclusions} 
\label{sec:conclusions}
We have derived expressions for the performance in terms of transient power losses of a droop-controlled inverter network subject to persistent small disturbances. We model the system with variable voltage dynamics, thus providing more realistic limits on performance than previous studies \cite{BamiehGayme2012, Tegling2014, Sjodin2014}. In particular, we show that previous results give a lower bound on performance and that transient losses are strictly larger if voltages are allowed to fluctuate. Furthermore, and in sharp contrast to previous results, these additional losses depend strongly on network topology. In fact, they will be larger in a highly connected network than in a loosely connected one. Our results also provide insights on how to tune controller parameters to improve this type of performance.

\section{Acknowledgements}
We would like to thank the anonymous reviewers for their valuable comments. 

\section*{Appendix}
\subsection*{Proof of Theorem \ref{thm:complete}}
Consider the function $\phi(x) = \frac{1}{\frac{1}{x}+k}$, which is concave for $x>0$, $k \ge 0$ ($\phi '' (x) = \frac{-2k}{(1+kx)^3} <0$). We have that $\textstyle  \lambda_n^B/c_Q >0$ for $n = 2,\ldots, N$ and can therefore apply Jensen's inequality of the form $\sum_{i = 1}^n \phi(x_i) \le n\phi\left(\frac{1}{n}\sum_{i = 1}^nx_i \right) $ to \eqref{eq:plossnorm} to obtain:
\begin{equation}
\label{eq:h2bound}
||H||_2^2  \le \frac{\alpha}{2k_P}(N-1) +\frac{\alpha}{2\tau_Q}(N-1)\frac{1}{\frac{c_Q}{\frac{1}{N-1}\sum_{n = 2}^N \lambda_n^B }+k_Q}.
\end{equation}
Using the definition of $L_B$ in \eqref{eq:defY}, we derive the average of the $N-1$ non-zero eigenvalues of $L_B$ as 
\[  \frac{1}{N-1}\sum_{n = 2}^N \lambda_n^B  =\frac{\mathrm{tr}\{L_B\}}{N-1} = \frac{2 \sum_{\mathcal{E}}b_{ik} }{N-1} = N \underline{b},\]
where $\underline{b}$ is the arithmetic mean of the susceptances of the $N(N-1)/2$ edges in the complete graph. Substituting the above into \eqref{eq:h2bound} yields the result \eqref{eq:completebound}.

Given that $\phi(x)$ is monotonically increasing in $x$, the inequality \eqref{eq:completeboundlower} is derived by setting  $L_B = b_{\min}L + \Delta L_B$. Here, $L$ is an unweighted complete graph Laplacian, and  $\Delta L_B$ is a complete graph Laplacian with edge weights $b_{ik} - b_{\min}\ge 0$. Since $L$ and $\Delta L_B$ are simultaneously diagonalizable \cite[Lemma A.1]{Sjodin2013}, $\lambda_n^B = b_{\min} N + \lambda^{\Delta B}_n \ge b_{\min}N$. If $L_B = b_{\min}L$, $\Delta L_B = 0$ and \eqref{eq:completeboundlower} holds with equality. 

\subsection*{Proof of Theorem \ref{thm:line}}
The argument follows the proof of Theorem~\ref{thm:complete}. Here, the average of the $N-1$ non-zero eigenvalues in \eqref{eq:h2bound} is
\[   \frac{1}{N-1}\sum_{n = 2}^N \lambda_n^B   = \frac{\mathrm{tr}\{L_B\}}{N-1} = \frac{2 \sum_{\mathcal{E}}b_{ij}  }{N-1} = 2 \underline{b},\]
where $\underline{b}$ is the mean of the $(N-1)$ edge susceptances in the line graph.

\bibliographystyle{IEEETran}
\bibliography{EmmasBib14}

\end{document}